\documentclass[12pt, a4paper]{amsart}

\usepackage{amsmath, amssymb, amsthm}
\usepackage{mathtools}
\usepackage[alphabetic]{amsrefs}
\usepackage{xcolor}
\usepackage{url}
\usepackage[colorlinks=true]{hyperref}
\usepackage{NichsCustom}

\usepackage[margin=3cm]{geometry}
\newtheorem{lemma}{Lemma}[section]
\newtheorem{theorem}[lemma]{Theorem}

\newtheorem{proposition}[lemma]{Proposition}
\theoremstyle{definition}
\newtheorem{definition}[lemma]{Definition}

\newtheorem{example*}[lemma]{Example}

\theoremstyle{remark}
\newtheorem{remark}[lemma]{Remark}

\makeatletter
\newtheorem*{rep@theorem}{\rep@title}
\newcommand{\newreptheorem}[2]{%
\newenvironment{rep#1}[1]{%
\def\rep@title{{\bf #2 \ref{##1}}}%
\begin{rep@theorem}}%
{\end{rep@theorem}}}
\makeatother
\newreptheorem{theorem}{Theorem}

\DeclareRobustCommand{\qedify}[1]{%
  \ifmmode \quad\hbox{#1}
  \else
    \leavevmode\unskip\penalty9999 \hbox{}\nobreak\hfill
    \quad\hbox{#1}%
  \fi
}
\newenvironment{example}{\begin{example*}\pushQED{\qedify{$\diamondsuit$}}}{\popQED\end{example*}}

\numberwithin{equation}{section}

\newcommand{\cD}{\mathcal{D}}

\newcommand{\cP}{\mathcal{P}}
\newcommand{\cH}{\mathcal{H}}

\newcommand{\cA}{\mathcal{A}}

\newcommand{\ZZ}{\mathbb{Z}}

\newcommand{\NN}{\mathbb{N}}
\newcommand{\RR}{\mathbb{R}}
\newcommand{\BB}{\mathbb{B}}

\newcommand{\bfx}{\mathbf{x}}

\newcommand{\bfu}{\mathbf{u}}
\newcommand{\bfv}{\mathbf{v}}
\newcommand{\bfw}{\mathbf{w}}

\newcommand{\bfa}{\mathbf{a}}
\newcommand{\bfb}{\mathbf{b}}
\newcommand{\bfzero}{\mathbf{0}}

\DeclareMathOperator{\rank}{rank}

\DeclareMathOperator{\trop}{trop}
\DeclareMathOperator{\supp}{supp}

\DeclareMathOperator{\spann}{span}

\DeclareMathOperator{\uMat}{\underline{Mat}}
\DeclareMathOperator{\si}{si}

\makeatletter
\newcommand{\superimpose}[2]{{\ooalign{$#1\@firstoftwo#2$\cr\hfil$#1\@secondoftwo#2$\hfil\cr}}}
\makeatother
\newcommand{\ttimes}{\hspace{0.3mm}{\mathpalette\superimpose{{\circ}{\cdot}}}\hspace{0.3mm}}
\newcommand{\tplus}{\oplus}

\newcommand{\Rbar}{\ensuremath{\overline{\mathbb R}}}

\newcommand{\defbold}{\textbf} 

\begin{document}

\title{Paving tropical ideals}

\author{Nicholas Anderson}
\address{School of Mathematical Sciences, Queen Mary University of
London, Mile End Road, London E1 4NS, United Kingdom}
\email{n.c.anderson@qmul.ac.uk}
%\thanks{}

\author{Felipe Rinc\'on}
\address{School of Mathematical Sciences, Queen Mary University of
London, Mile End Road, London E1 4NS, United Kingdom}
\email{f.rincon@qmul.ac.uk}
%\thanks{}
\maketitle

\begin{abstract}
Tropical ideals are a class of ideals in the tropical polynomial semiring that combinatorially abstracts the possible collections of supports of all polynomials in an ideal over a field. We study zero-dimensional tropical ideals $I$ with Boolean coefficients in which all underlying matroids are paving matroids, or equivalently, in which all polynomials of minimal support have support of size $\deg(I)$ or $\deg(I)+1$ -- we call them paving tropical ideals. We show that paving tropical ideals of degree $d+1$ are in bijection with $\ZZ^n$-invariant $d$-partitions of $\ZZ^n$. This implies that zero-dimensional tropical ideals of degree $3$ with Boolean coefficients are in bijection with $\ZZ^n$-invariant $2$-partitions of quotient groups of the form $\ZZ^n/L$. We provide several applications of these techniques, including a construction of uncountably many zero-dimensional degree-$3$ tropical ideals in one variable with Boolean coefficients, and new examples of non-realizable zero-dimensional tropical ideals.
\end{abstract}

\section{Introduction}

Investigating the connections between algebraic geometry and combinatorics has been an exceptionally fruitful pursuit in many fields over the past few decades. Tropical geometry has served as a very useful tool in this endeavor, allowing the study of subvarieties of toric varieties by analyzing their corresponding tropicalizations -- finite polyhedral complexes that retain important information about the original algebraic varieties. 
As explored in \cite{GG16} and \cite{MR16}, this tropicalization process can also be carried out algebraically, by sending the defining ideal of a subvariety to an ideal in the semiring of tropical polynomials.

Tropical ideals, axiomatized by Maclagan and Rinc\'on in \cite{MR18}, have been proposed as a sensible class of tropical objects for developing the study of tropical scheme theory. The class of tropical ideals includes all tropicalizations of classical ideals, but it is in general much larger. In \cite{MR18}, and later in \cite{MR}, Maclagan and Rinc\'on established several essential properties of tropical ideals and their associated tropical varieties, analogous to those of ideals in a polynomial ring and their algebraic varieties.
These include the fact that tropical ideals satisfy the ascending chain condition, have a Hilbert function that is eventually polynomial -- and thus a notion of dimension and degree -- and define varieties that are finite balanced polyhedral complexes.
While these properties were known for tropicalizations of ideals of a polynomial ring, called realizable tropical ideals, their proofs in the non-realizable case require new careful combinatorial analyses.
Part of this difficulty is related to the fact that we know very few examples so far of non-realizable tropical ideals: while it is easy to construct ideals over rings, constructing tropical ideals in a purely tropical setting remains a more difficult task, as, for example, they are not closed under sums or intersections, nor do they admit a notion of finite generation that we know of.

Studying zero-dimensional tropical ideals has been a first step in the exploration of tropical ideals. These ideals play a very important role in the theory, as often general properties of tropical ideals can be reduced to the zero-dimensional case. In addition, their study is directly connected to an interesting moduli space -- the (tropical) Hilbert scheme of points. In her PhD thesis \cite{Zaj18}, Zajaczkowska studied degree-$2$ zero-dimensional tropical ideals with coefficients in the Boolean semiring $\BB$, and showed that they are in one-to-one correspondence with sublattices of $\Z^n$. This result was used in \cite{FGG} to study the space of all degree-$2$ zero-dimensional homogeneous tropical ideals in two variables, what can be thought of as the tropical Hilbert scheme of $2$ points in $\trop(\mathbb P^1)$. Higher-degree zero-dimensional tropical ideals have been studied in \cite{Sil} and \cite{FGG} in the case of homogeneous ideals in two variables, where interesting connections to Schur polynomials were developed in the realizable case.   

In this paper we study the class of zero-dimensional tropical ideals for which all underlying matroids are paving matroids, which we call \textit{paving tropical ideals}. Equivalently, a zero-dimensional tropical ideal $I$ is a paving tropical ideal if and only if all polynomials of minimal support in $I$ have support of size either $\deg(I)$ or $\deg(I)+1$. We give more detailed definitions in Section \ref{sec:paving}.

Throughout the paper we focus on the case of paving tropical ideals with coefficients in the Boolean semiring $\BB$. This allows us to focus purely on the combinatorics of the underlying matroids, and provides already many interesting examples. Understanding tropical ideals with Boolean coefficients is an important and useful step, as general tropical ideals can be studied using them by taking initial ideals or trivializing the valuation. 

Paving matroids of rank $d+1$ over a ground set $E$ are in bijection with $d$-partitions of $E$ -- a generalization of the notion of partition that allows for intersections of size less than $d$ between subsets. We extend this bijection to prove the following result. 

\newtheorem*{thm:correspondenceThm}{\bf Theorem \ref{thm:correspondenceThm}}
\begin{thm:correspondenceThm}
There is a one-to-one correspondence between paving tropical ideals of degree $d+1$ with coefficients in the Boolean semiring $\BB$ and $d$-partitions of $\Z^n$ that are invariant under the action of $\Z^n$.
\end{thm:correspondenceThm}
 
Using the fact that any matroid of rank $2$ with no loops is a paving matroid, we recover the correspondence described in \cite{Zaj18}*{Theorem 4.2.4} between zero-dimensional degree-$2$ tropical ideals with Boolean coefficients and sublattices of $\ZZ^n$.

We also use these techniques to understand general zero-dimensional tropical ideals of degree $3$ with Boolean coefficients. Based on the fact that the simplification of any rank $3$ matroid is a paving matroid, we prove the following result.

\newtheorem*{thm:simplification}{\bf Theorem \ref{thm:simplification}}
\begin{thm:simplification}
There is a one-to-one correspondence between zero-dimensional tropical ideals of degree $3$ with Boolean coefficients and pairs $(L,\cP)$, where $L \subset \ZZ^n$ is a sublattice and $\cP$ is a $\ZZ^n$-invariant $2$-partition of the quotient group $\Z^n/L$.
\end{thm:simplification}

Finally, we use the fact that $d$-partitions have a natural notion of ``generating sets'' to construct explicit examples of tropical ideals. This allows us to provide the following applications:

\begin{itemize}
    \item {\bf Proposition \ref{prop:manyPavingIdeals}.} There are uncountably many zero-dimensional tropical ideals of degree 3 in $\BB[x^{\pm1}]$.
    \item {\bf Theorem \ref{thm:matroidIdeals}.} Any paving matroid on the set of variables $\{x_1, \dots, x_n\}$ can be extended to a zero-dimensional tropical ideal $I \subset \BB[x_1^{\pm 1}, \dots, x_n^{\pm 1}]$. This can be used to construct many examples of non-realizable tropical ideals. 
    \item {\bf Proposition \ref{prop:nonrealizable}.} There is a non-realizable zero-dimensional degree-$2$ tropical ideal in $\BB[x_1^{\pm1}, x_2^{\pm 1}, x_3^{\pm 1}]$.
\end{itemize}

The paper is organized as follows. In Section \ref{sec:paving} we provide all the basic definitions, and we develop the correspondence between paving tropical ideals and $d$-partitions stated in Theorem \ref{thm:correspondenceThm}. In Section \ref{sec:lowrank} we use these ideas to investigate general zero-dimensional tropical ideals of degree at most 3, proving Theorem \ref{thm:simplification}. Lastly, in Section \ref{sec:applications} we develop the applications listed above.

The results presented in this paper were obtained as part of the first author's MSc dissertation at Queen Mary University of London, under the supervision of the second author.

\section{Paving tropical ideals}\label{sec:paving}

We write 
$\Rbar = (\RR \cup \{\infty\}, \tplus, \ttimes)$ 
for the \defbold{tropical semiring}, where the tropical sum 
$\oplus$ is $\mathrm{min}$ and the tropical multiplication
$\ttimes$ is $+$. 
The \defbold{Boolean semiring} is 
$\BB := \{\infty, 0\} \subset \Rbar.$

In this paper we focus on tropical ideals with coefficients in $\BB$.
Given a polynomial $f = \bigoplus_{\bfu \in \ZZ^n} c_\bfu \bfx^\bfu 
\in \BB[x_1^{\pm 1}, \dots, x_n^{\pm 1}]$, its \defbold{support} is
$$\supp(f) := \{ \bfu \in \ZZ^n : c_\bfu \neq \infty\}.$$

\begin{definition}%\label{d:tropideal_B}
An ideal $I$ in $\BB[x_1, \dots, x_n]$ or $\BB[x_1^{\pm 1}, \dots, x_n^{\pm 1}]$ 
is a \defbold{tropical ideal} if it satisfies the following 
``monomial elimination axiom'': 

$\bullet$ For any $f,g \in I$ and any $\bfu \in \supp(f) \cap \supp(g)$, 
there exists $h \in I$ such that
$$
\supp(f) \,\Delta\, \supp(g) \,\subset\, \supp(h) \,\subset\, 
(\supp(f) \cup \supp(g)) - \{\bfu\},
$$
where $\Delta$ denotes symmetric difference of sets. 
Equivalently, $I$ is a tropical ideal if
for any $E \subset \ZZ^n$ finite, the collection
$$
\supp(I|_E) := \{ \supp(f) : f \in I \text{ and } \supp(f) \subset E \} \subset 2^E
$$
is the collection of cycles (i.e. unions of circuits) 
of a matroid on the ground set $E$.
We denote this matroid by $\uMat(I|_E)$.
\end{definition}

Tropical ideals are combinatorial generalizations of the collections
of supports of all polynomials in an ideal with coefficients in a field.
Indeed, if $K$ is a field and $F \in K[x_1^{\pm 1}, \dots, x_n^{\pm 1}]$,
its \defbold{tropicalization} is the tropical polynomial
$$ \trop(F) := \bigoplus_{\bfu \in \supp(F)} \bfx^\bfu \, \in \BB[x_1^{\pm 1}, \dots, x_n^{\pm 1}].$$
If $J \subset K[x_1^{\pm 1}, \dots, x_n^{\pm 1}]$ is an ideal, 
its tropicalization
$$\trop(J) :=  \langle \supp(F) : F \in I \rangle \, \subset \BB[x_1^{\pm 1}, \dots, x_n^{\pm 1}]$$
is a tropical ideal. A tropical ideal of the form $\trop(J)$ for $J \subset K[x_1^{\pm 1}, \dots, x_n^{\pm 1}]$ is called \defbold{realizable} over the field $K$.

Tropical ideals have many properties analogous 
to those of ideals in a polynomial ring over a field.

The \defbold{Hilbert function} of a homogeneous tropical ideal 
$I \subset \BB[x_0, \dots, x_n]$ is the map $H_I : \NN \to \NN$ given by
$H_I(d) = \rank(\uMat(I_d))$, where $I_d$ is the degree-$d$ part of $I$.
In other words, $H_I(d)$ is the maximal size of a subset $B$ of monomials of 
degree $d$ with the property that $B$ does not contain the support of any 
polynomial in $I_d$. 

This notion can be extended to arbitrary tropical ideals in $\BB[x_1, \dots, x_n]$,
and also to tropical ideals in $\BB[x_1^{\pm 1}, \dots, x_n^{\pm 1}]$, 
as we now explain. 
The \defbold{homogenization} of a tropical polynomial $f = \bigoplus c_\bfu
\bfx^\bfu \in \BB[x_1,\dots,x_n]$ is 
$$\tilde{f} = \textstyle \bigoplus x_0^{d-|\bfu|} \ttimes c_\bfu
\bfx^\bfu \in \BB[x_0,x_1,\dots,x_n],$$
where $|\bfu| := \sum_{i=1}^n u_i$ and 
$d = \max(|\bfu| : c_\bfu \neq \infty )$.  
The homogenization of an ideal $I \subset \BB[x_1,\dots,x_n]$ 
is the homogeneous ideal 
$$I^h := \langle \tilde{f} : f \in I \rangle \subset \BB[x_0,x_1,\dots,x_n].$$
If $I$ is a tropical ideal then $I^h$ is a tropical ideal as well \cite{MR}*{Lemma 2.1}. 
In this case, the Hilbert function of $I$ is defined as the Hilbert function of $I^h$.
If $J \subset \BB[x_1^{\pm 1},\dots,x_n^{\pm 1}]$ is a tropical ideal, 
the intersection $J \cap \BB[x_1,\dots,x_n]$
is a tropical ideal in $\BB[x_1,\dots,x_n]$, 
and the Hilbert function of $J$ is defined as the Hilbert function of 
$J \cap \Rbar[x_1,\dots,x_n]$.

In all these cases, the Hilbert function of a tropical ideal $I$
agrees with a polynomial $P_I(d)$ for $d \gg 0$, called the 
\defbold{Hilbert polynomial} of $I$. 
The \defbold{dimension} $\dim(I)$ of $I$ 
is defined as the degree of $P_I$, and the \defbold{degree} $\deg(I)$ of $I$ is
defined as $\dim(I)!$ times the leading coefficient of $P_I$.

If $I$ is a zero-dimensional tropical ideal, its Hilbert polynomial is a constant, 
and the degree of $I$ is equal to that constant. The degree of a zero-dimensional
ideal $I$ is equal to the maximal size of a subset $B$ of monomials 
with the property that $B$ does not contain the support of any polynomial in $I$;
see \cite{MR}*{Lemma 5.2}.
   
A \defbold{finitary matroid} on a (possibly infinite) ground set $E$ is a non-empty collection of subsets of $E$, called independent sets, which is closed under taking subsets, satisfies the usual augmentation axiom for matroids, and furthermore, if all
finite subsets of a subset $I$ are independent then $I$ is independent.
Equivalently, minimal dependent sets (called circuits) are non-empty, non-comparable, satisfy the usual circuit elimination axiom for matroids, and are finite.
Maximal independent subsets are called bases.
A finitary matroid is said to have \defbold{finite rank} if all bases are finite; 
in this case, all bases have the same cardinality, called the rank of the matroid.
For more on different cryptomorphisms for finite-rank matroids, finitary matroids, and more general infinite matroids, see for example 
\cite{White86}*{Section 2.4}, \cite{White92}*{Chapter 3}, and \cite{InfiniteMatroids}.

A tropical ideal $I$ is naturally encoded by a finitary matroid on the set of monomials, as described below.

\begin{definition}
Any tropical ideal $I$ in $\BB[x_1, \dots, x_n]$ or $\BB[x_1^{\pm 1}, \dots, x_n^{\pm 1}]$ has an \defbold{underlying finitary matroid} $\uMat(I)$, whose ground set $E$ is equal to
the set $\NN^n$ or $\ZZ^n$, respectively. A subset $A \subset E$ is 
independent in $\uMat(I)$ if it does not contain the support of any polynomial in $I$.
Equivalently, the circuits of $\uMat(I)$ are the minimal supports of polynomials in $I$. 
If the tropical ideal $I$ is zero-dimensional then this matroid has finite rank, equal to $\deg(I)$.
In this case, the rank of a subset $X \subset E$ is 
$$\rank(X) := \max \{ |A| : A \subset X \text{ is independent in } \uMat(I)\} < \infty.$$ 
\end{definition}

In this paper we investigate zero-dimensional tropical ideals whose underlying 
matroids are paving matroids.

\begin{definition} A (possibly infinite) matroid $M$ of finite rank $\rank(M) < \infty$
is a \defbold{paving matroid} if all the circuits of $M$ have size $\rank(M)$ or $\rank(M)+1$.
\end{definition}

For a finite rank matroid, the property of being paving is a local property, in the following sense.

\begin{proposition}\label{prop:pavingMinors}
A (possibly infinite) matroid $M$ of finite rank on the ground set $E$ is a paving matroid if and only if for all $E' \subset E$ finite, the restriction
$M|_{E'}$ is a paving matroid.
\end{proposition}
\begin{proof}
Suppose $M$ is a paving matroid of rank-$d$ on the set $E$. For any finite subset $E'\subset E$, the circuits of $M|_{E'}$ are the circuits of $M$ contained within $E'$. Since all the circuits of $M$ are of size $d$ or $d+1$, this also holds for $M_{E'}$.\newpar
Conversely, suppose that $M$ is a rank $d$ matroid for which every finite restriction $M|_{E'}$ is a paving matroid, but $M$ is not a paving matroid. Then $M$ contains a circuit $C$ of size at most $d-1$. 
Let $B$ be a basis of $M$. The restriction $M|_{B\cup C}$ then has rank $d$, and has a circuit of size $d-1$. This means that $M|_{B\cup C}$ is not a paving matroid, a contradiction.
\end{proof}

\begin{definition}
A zero-dimensional tropical ideal $I$ is a \defbold{paving tropical ideal} if
the underlying matroid $\uMat(I)$ is a paving matroid, i.e., 
all polynomials of minimal support in $I$ have support of size equal 
to $\deg(I)$ or $\deg(I) + 1$. Equivalently, $I$ is a paving tropical ideal if and only if 
the matroid $\uMat(I|_E)$ is a paving matroid for any finite set of monomials $E$.
\end{definition}

Finite paving matroids of rank $d+1$ on a set $E$ are known to be in correspondence with 
collections of subsets of $E$ called $d$-partitions.
We now extend this correspondence to include infinite paving matroids.

\begin{definition} 
Let $E$ be a (possibly infinite) set with at least $d+1$ elements. 
A \defbold{$d$-partition} of $E$ is a collection $\mathcal P$ of subsets of $E$ 
satisfying
\begin{enumerate}
    \item[(P1)] $|\mathcal P|\geq 2$,
    \item[(P2)] $|S|\geq d$ for any $S \in \mathcal P$, and
    \item[(P3)] each $d$-element subset of $E$ appears in a unique $S \in \mathcal P$.
\end{enumerate}
A subset $S \in \mathcal P$ is called a \defbold{block} of $\mathcal P$.
\end{definition}

Note that $d$-partitions are a generalization of the usual notion of partition of a set;
indeed, a partition of a set $E$ is the same as a $1$-partition of $E$ (with the exception of the 
trivial partition $\{E\}$).

A \defbold{hyperplane} of a finite-rank matroid $M$ is a maximal non-spanning subset of $M$, i.e., a maximal subset of rank $\rank(M)-1$.
We denote the collection of hyperplanes of $M$ by $\cH(M)$; this collection is enough to determine the matroid $M$.

For completeness, we recall the axioms that characterize collections of hyperplanes of
finite-rank matroids; see for example \cite{White86}*{Proposition 2.4.2}.
\begin{proposition}
A collection $\cH$ of subsets of a (possibly infinite) set $E$ is the collection of hyperplanes of a finite-rank
matroid over $E$ if and only if $\cH$ satisfies
\begin{enumerate}
    \item[(H1)] $E \notin \cH$.
    \item[(H2)] Any two subsets in $\cH$ are incomparable.
    \item[(H3)] For every $H_1, H_2 \in \cH$ with $H_1 \neq H_2$, and every $x \in E$,
    there exists $H_3 \in \cH$ such that $(H_1 \cap H_2) \cup \{x\} \subset H_3$.
    \item[(HF)] If a subset $X \subset E$ satisfies $X \not\subset H$ for all $H \in \cH$, 
    then there exists a finite $X' \subset X$ satisfying $X' \not\subset H$ for all $H \in \cH$.
\end{enumerate}
\end{proposition}

The following correspondence is well-known in the case of matroids on a finite ground set,
and is sometimes provided as the definition of a paving matroid \cite{Welsh76}*{Chapter 2.3}.

\begin{proposition}\label{prop:matroidCorrespondence}
There is a one-to-one correspondence between rank-$(d+1)$ paving matroids 
on a (possibly infinite) set $E$ and $d$-partitions of $E$. 
Concretely, if $M$ is a paving matroid on $E$ of rank $d+1$ then 
its collection of hyperplanes $\cH(M)$ is a $d$-partition of $E$. 
Conversely, for any $d$-partition $\cP$ of $E$ there is a paving matroid $M(\cP)$ on $E$ 
of rank $d+1$ whose collection of hyperplanes is equal to $\cP$.
\end{proposition}
\begin{proof}
Suppose that $M$ is a paving matroid on $E$ of rank $d+1$. 
Hyperplanes of $M$ are flats of rank $d$. 
As $M$ has rank at least $2$, the collection of hyperplanes $\cH(M)$ satisfies the $d$-partition axioms (P1) and (P2). 
To show that $\cH(M)$ satisfies (P3), take $S \subset E$ of size $d$. 
Since all the circuits of $M$ have size $d+1$ or $d+2$, the subset $S$ is independent in $M$.
The only hyperplane that contains $S$ is thus $\spann(S)$, showing that $\cH(M)$ satisfies (P3).

Conversely, suppose that $\cP$ is a $d$-partition of $E$. Axiom (P3) ensures that any
two subsets in $\cP$ are incomparable, and thus $\cP$ satisfies the hyperplane axioms
(H1) and (H2). 
To show that $\cP$ satisfies (H3), take distinct $S_1,S_2 \in \cP$, and $x \in E$.
By axiom (P3) we have that $|(S_1 \cap S_2) \cup \{x\}| \leq d$, and thus 
$(S_1 \cap S_2) \cup \{x\}$ must be contained in some $S_3 \in \cP$. 
For axiom (HF), we must show that if $X \subset E$ is infinite and $X \not\subset S$ for all $S\in \cP$ then $X$ contains a finite subset $X'$ with this property. 
Take $Y \subset X$ of size $d$. By (P3), we have $Y \subset S$ for a unique $S \in \cP$.
Since $X \not\subset S$, there exists $x \in X \setminus S$. The set $X' := Y \cup \{x\}$ 
is thus not contained in any $S' \in \cP$, as otherwise $Y$ would be contained in both $S$
and $S'$.
This shows that $\cP$ is the collection of hyperplanes of a finite-rank matroid $M$ on $E$.  
Finally, the fact that every $d$-subset of $E$ is contained in a unique hyperplane of $M$
implies that $M$ has rank $d+1$ and every $d$-subset of $E$ is independent in $M$,
which means that $M$ is a paving matroid.
\end{proof} 
 
We now use the previous general correspondence between paving matroids and $d$-partitions 
to encode paving tropical ideals as special $d$-partitions of $\ZZ^n$. 

The action of the additive group $\ZZ^n$ on itself induces an action of $\ZZ^n$ on 
subsets of $\ZZ^n$, namely, if $\bfu \in \ZZ^n$ and $S \subset \ZZ^n$ then  
$$\bfu + S := \{ \bfu + \bfv : \bfv \in S\}.$$
This in turn induces an action of $\ZZ^n$ on collections of subsets
of $\ZZ^n$: If $\bfu \in \ZZ^n$ and $\cP \subset 2^{\ZZ^n}$, the collection
$$\bfu + \cP := \{ \bfu + S : S \in \cP\}$$
is another collection of subsets of $\ZZ^n$.
We say that $\cP \subset 2^{\ZZ^n}$ is \defbold{invariant} under the action of $\ZZ^n$ if $\bfu + \cP = \cP$ for all $\bfu \in \ZZ^n$.
In a similar way, the action of $\ZZ^n$ on itself induces a natural action of $\ZZ^n$ on finitary matroids on the ground set $\ZZ^n$.

\begin{theorem}\label{thm:correspondenceThm}
There is a bijection between tropical ideals in 
$\BB[x_1^{\pm 1},\dots,x_n^{\pm 1}]$ and finitary matroids on $\ZZ^n$ that 
are invariant under the action of $\ZZ^n$, sending a tropical ideal $I$
to its underlying matroid $\uMat(I)$.

Consequently, there is a bijection between degree-$(d+1)$ 
paving tropical ideals in $\BB[x_1^{\pm 1},\dots,x_n^{\pm 1}]$ 
and $d$-partitions of $\ZZ^n$ that are invariant under the action of $\ZZ^n$, 
sending a paving tropical ideal $I$ to the collection of hyperplanes $\cH(\uMat(I))$ 
of its underlying matroid.
\end{theorem}
\begin{proof}
If $I$ is a tropical ideal in $\BB[x_1^{\pm 1},\dots,x_n^{\pm 1}]$ then its underlying 
matroid $\uMat(I)$ is a finitary matroid on the set $\ZZ^n$. Since $I$ is invariant
under multiplication by any Laurent monomial, the matroid $\uMat(I)$ is invariant under 
the action of $\ZZ^n$. In addition, note that any $\ZZ^n$-invariant finitary matroid $M$ on $\ZZ^n$ 
is the underlying matroid of a unique tropical ideal $I \subset \BB[x_1^{\pm 1},\dots,x_n^{\pm 1}]$, namely,
the ideal consisting of polynomials of the form $f= \bigoplus_{\bfu \in C} \bfx^\bfu$
with $C$ a cycle (i.e. union of circuits) of $M$. This proves the first part of the theorem.

By definition, a tropical ideal $I \subset \BB[x_1^{\pm 1},\dots,x_n^{\pm 1}]$ is a degree-$(d+1)$ paving tropical ideal if and only if $\uMat(I)$ is a $\ZZ^n$-invariant rank-$(d+1)$ paving matroid on $\ZZ^n$. By Proposition \ref{prop:matroidCorrespondence}, 
such matroids are in bijection with $\ZZ^n$-invariant $d$-partitions of $\ZZ^n$, via the map that sends a matroid $\uMat(I)$ to its collection of hyperplanes $\cH(\uMat(I))$. 
\end{proof}

\begin{remark}\label{rem:concrete}
It is worth mentioning without using any matroid terminology what the correspondence in Theorem
\ref{thm:correspondenceThm} is in the case of paving tropical ideals.
If $\cP$ is a $\ZZ^n$-invariant $d$-partition of $\ZZ^n$, 
the degree-$(d+1)$ paving tropical ideal $I \subset \BB[x_1^{\pm 1},\dots,x_n^{\pm 1}]$ corresponding to $\cP$ 
has minimal-support polynomials of the form $f = \bigoplus_{\bfu \in C} \bfx^\bfu$, 
where $C \subset \ZZ^n$ has size $d+1$ and is contained in a block $S \in \cP$, 
or $C$ has size $d+2$ and does not contain any such $(d+1)$-subset.
\end{remark}

\begin{remark}
Tropical ideals in $\BB[x_1^{\pm 1},\dots,x_n^{\pm 1}]$ are in correspondence with tropical 
ideals in $\BB[x_1,\dots,x_n]$ that are saturated with respect to the product of the variables, 
i.e., tropical ideals $I \subset \BB[x_1,\dots,x_n]$ for which $x_i \ttimes f \in I$ implies $f \in I$. 
Specifically, if $I \subset \BB[x_1^{\pm 1},\dots,x_n^{\pm 1}]$ is a tropical ideal then
the intersection $J := I \cap \BB[x_1,\dots,x_n]$ is a saturated tropical ideal in $\BB[x_1,\dots,x_n]$. 
Conversely, if $J \subset \BB[x_1,\dots,x_n]$ is a saturated tropical ideal then the ideal 
$J \, \BB[x_1^{\pm 1},\dots,x_n^{\pm 1}]$ it generates in $\BB[x_1^{\pm 1},\dots,x_n^{\pm 1}]$ is also a tropical ideal, and these two maps are inverses to each other.
More details about this correspondence can be found in \cite{MR}*{Lemma 2.1}.
Theorem \ref{thm:correspondenceThm} thus implies that saturated degree-$(d+1)$ paving tropical ideals in
$\BB[x_1,\dots,x_n]$ are also in bijection with $\ZZ^n$-invariant $d$-partitions of $\ZZ^n$.

In a similar way, tropical ideals in $\BB[x_1^{\pm 1},\dots,x_n^{\pm 1}]$ are also in 
correspondence with homogeneous tropical ideals in $\BB[x_0,\dots,x_n]$ that are 
saturated with respect to the product of all the variables; for details see \cite{MR}*{Lemma 2.1} as well. Again,
Theorem \ref{thm:correspondenceThm} implies that homogeneous saturated degree-$(d+1)$ paving tropical ideals in $\BB[x_0,\dots,x_n]$ are also in
bijection with $\ZZ^n$-invariant $d$-partitions of $\ZZ^n$.
\end{remark}

We conclude this section with a result about which subsets of $\ZZ^n$ can be hyperplanes of a paving tropical ideal. 

\begin{definition}\label{def:sparseHyperplane}
We say a subset $S \subset \Z^n$ is \textbf{$d$-sparse} if there exists no $\bfu \in \ZZ^n\setminus\{\bfzero\}$ such that $|S\cap (\bfu + S)|\geq d$.
Equivalently, $S \subset \Z^n$ is $d$-sparse if whenever $\bfa_1, \dots, \bfa_d, \bfb_1, \dots ,\bfb_d \in S$ satisfy $\bfzero \neq \bfa_1-\bfb_1 = \bfa_2-\bfb_2 = \dots = \bfa_d-\bfb_d$ then $(\bfa_i,\bfb_i)=(\bfa_j,\bfb_j)$ for some $i\neq j$.
\end{definition}

\begin{proposition}\label{prop:latticeBlocks}
Suppose $\cP$ is a $\ZZ^n$-invariant $d$-partition of $\ZZ^n$.
Then any block $S \in \cP$ is either $d$-sparse or a non-trivial affine sublattice of $\ZZ^n$,
i.e. it has the form $S = \bfv + L$ for $\bfv \in \ZZ^n$ and $\{\bfzero\} \subsetneq L \subsetneq \ZZ^n$ a sublattice.
\end{proposition}
\begin{proof}
Fix a block $S \in \cP$. As $\cP$ is $\ZZ^n$-invariant, for any $\bfu \in \ZZ^n$
we have $\bfu + S \in \cP$. Note that axioms (P2) and (P3) for $d$-partitions imply that 
$\bfu + S = S$ if and only if $|S\cap (\bfu+S)|\geq d$. 
The stabilizer of $S$ under the action of $\ZZ^n$ on the blocks of $\cP$ is then 
$$L := \{ \bfu \in \ZZ^n : \bfu + S = S \} = \{ \bfu \in \ZZ^n : |S\cap (\bfu + S)|\geq d \},$$ 
which is a subgroup of $\ZZ^n$. 
The case when $L$ is the trivial lattice $L = \{\bfzero\}$ 
is exactly the case when $S$ is $d$-sparse. 
If $L$ is non-trivial, fix $\bfv \in S$. 
Note that $S \supset \bfv + L$. For any $\bfw \in S$, we also have $S \supset \bfw + L$,
and thus $(\bfv - \bfw) + S \supset \bfv + L$. 
This means that $S \cap ((\bfv - \bfw) + S) \supset \bfv + L$, 
so $|S \cap ((\bfv - \bfw) + S)| \geq |\bfv + L| = \infty \geq d$ and therefore $\bfv - \bfw \in L$. 
It follows that $S = \bfv + L$, as desired.
\end{proof}

In view of Theorem \ref{thm:correspondenceThm}, Proposition \ref{prop:latticeBlocks}
says that any hyperplane of a degree-$(d+1)$ paving tropical ideal is either $d$-sparse
or an affine sublattice of $\ZZ^n$.
In Section \ref{sec:applications} we will provide examples and applications based
on constructing paving matroids with specified $d$-sparse hyperplanes.

\section{Zero-dimensional tropical ideals of low degree}\label{sec:lowrank}

In this section we use the setup developed in Section \ref{sec:paving} to better understand
zero-dimensional tropical ideals in $\B[x_1^{\pm 1},\dots,x_n^{\pm 1}]$ of degree at most $3$.  

The case of zero-dimensional tropical ideals of degree at most $2$ 
was studied by Zajaczkowska in her PhD thesis \cite{Zaj18}. 
There she showed that all zero-dimensional tropical ideals in $\Rbar[x_1^{\pm 1},\dots,x_n^{\pm 1}]$ 
of degree $1$ are realizable, and are of the form $\trop(\langle x_1 - a_1, \dots, x_n- a_n \rangle)$ where
$a_1, \dots, a_n$ are elements of a valued field $K$. (Note that this statement includes the case of tropical ideals
with non-Boolean coefficients.) In fact, these ideals are exactly the maximal tropical ideals
of $\Rbar[x_1^{\pm 1},\dots,x_n^{\pm 1}]$, as shown in \cite{MR18}*{Example 5.19}.

Zero-dimensional tropical ideals of degree $2$ with Boolean coefficients were 
shown in \cite{Zaj18} to be in correspondence with sublattices of $\ZZ^n$. It was also
proved that in the case of $n=1$ they are all realizable, although as we will see in the next section,
this is not always the case if $n > 1$. 
The case of zero-dimensional tropical ideals of degree $2$ in one variable with coefficients in $\Rbar$
was studied in depth in \cite{FGG}, where the authors show that, while these tropical
ideals are naturally parametrized by points in a certain infinite-dimensional simplicial complex, 
the realizable ones correspond only to a 1-dimensional subcomplex.

The correspondence between zero-dimensional tropical ideals of degree $2$ and 
sublattices of $\ZZ^n$ follows directly from our results in Section \ref{sec:paving},
as we now explain.

\begin{proposition}\label{prop:degree2}
There is a bijection between zero-dimensional degree-$2$ tropical ideals in 
$\B[x_1^{\pm 1},\dots,x_n^{\pm 1}]$ and proper sublattices $L \subsetneq \ZZ^n$. Under this
correspondence, the tropical ideal associated to a sublattice $L$ has minimal-support polynomials
of the form $\bfx^\bfu \tplus \bfx^\bfv$ with $\bfu - \bfv \in L$, and 
$\bfx^\bfu \tplus \bfx^\bfv \tplus \bfx^\bfw$ with no pairwise difference between $\bfu,\bfv,\bfw$ in $L$. 
\end{proposition}

\begin{proof}
Note that zero-dimensional degree-$2$ tropical ideals of $\B[x_1^{\pm 1},\dots,x_n^{\pm 1}]$ 
contain no monomials, and thus they are all paving tropical ideals. 
It follows from Theorem \ref{thm:correspondenceThm} that they are in bijection with $\ZZ^n$-invariant
$1$-partitions (i.e. non-trivial partitions) of $\Z^n$. 
Suppose $\cP$ is a $\ZZ^n$-invariant partition of $\Z^n$, and let $L \in \cP$ be the block containing $\bfzero$. 
If $L$ has size at least $2$ then $L$ is not $1$-sparse, so by Proposition \ref{prop:latticeBlocks}, $L$ is a sublattice of $\ZZ^n$. 
The same is true, of course, if $|L| = 1$.
As $\cP$ is a partition that is invariant under the action of $\ZZ^n$, it follows that
$\cP$ consists of all the translates of $L$.
This shows that non-trivial $\ZZ^n$-invariant partitions of $\Z^n$ are in bijection with proper sublattices of $\ZZ^n$.
Finally, the concrete correspondence in the statement follows from Remark \ref{rem:concrete}.
\end{proof}

This approach can be extended to study all zero-dimensional degree-$3$ tropical ideals with Boolean coefficients, based on the fact that all rank-$3$ matroids without loops are paving matroids once parallel elements are identified.

Suppose $M$ is a finitary matroid on the ground set $E$ with no loops. 
Two elements $a,b \in E$ are called \defbold{parallel} in $M$, denoted $a \sim b$, if $a=b$ or $\{a,b\}$ is a circuit of $M$. 
This parallelism relation $\sim$ is an equivalence relation on $E$. 
The fact that a subset $\{a_1, \dots, a_k\} \subset E$ containing no parallel elements is a circuit of $M$ depends only on the equivalence classes $[a_i]$ of the $a_i$, that is, $\{a_1, \dots, a_k\}$ is a circuit of $M$ if and only if $\{b_1, \dots, b_k\}$ is a circuit of $M$ whenever $b_i \sim a_i$ for all $i$. 
The \defbold{simplification} of the matroid $M$, denoted by $\si(M)$, is the finitary matroid whose ground set is the set of equivalence classes $E/\!\sim$, and whose circuits are the subsets $\{[a_1], \dots, [a_k]\}$ for which $\{a_1, \dots, a_k\}$ is a circuit of $M$.

\begin{definition}
Let $I$ be a tropical ideal in $\BB[x_1^{\pm 1},\dots,x_n^{\pm 1}]$. Its \defbold{binomial lattice} is the sublattice of $\ZZ^n$ given by
\[L_I := \{ \bfu - \bfv \in \ZZ^n : \bfx^\bfu \tplus \bfx^\bfv \in I\} \cup \{\bfzero\}. \]
The fact that $L_I$ is indeed a sublattice follows directly from the monomial elimination axiom for $I$.
\end{definition}

\begin{theorem}\label{thm:simplification}
There is a one-to-one correspondence between zero-dimensional tropical ideals in $\BB[x_1^{\pm 1},\dots,x_n^{\pm 1}]$ of degree 3 and pairs $(L, \cP)$, where $L$ is a sublattice of $\ZZ^n$ and $\cP$ is a $2$-partition of the quotient group $\ZZ^n/L$ which is invariant under the action of $\ZZ^n$.
This correspondence sends a zero-dimensional tropical ideal $I$ of degree $3$ to the pair $(L_I, \cH(\si(\uMat(I))))$.
\end{theorem}
\begin{proof}
Suppose $I$ is a zero-dimensional degree-$3$ tropical ideal in $\BB[x_1^{\pm 1},\dots,x_n^{\pm 1}]$. The underlying matroid $\uMat(I)$ has no loops, and its parallelism classes are the cosets of the binomial lattice $L_I$ in $\ZZ^n$. 
The simplification $\si(\uMat(I))$ is then a rank-$3$ paving matroid on the quotient subgroup $\ZZ^n/L_I$. As $\uMat(I)$ is invariant under the action of $\ZZ^n$, so is $\si(\uMat(I))$.
By Proposition \ref{prop:matroidCorrespondence}, the collection of hyperplanes $\cH(\si(\uMat(I)))$ is a $2$-partition of $\ZZ^n/L_I$ that is invariant under the action of $\ZZ^n$. Note that the matroid $\uMat(I)$, and thus the tropical ideal $I$, is determined by the pair $(L_I, \cH(\si(\uMat(I))))$.

Conversely, suppose $L$ is a sublattice of $\ZZ^n$ and $\cP$ is a $\ZZ^n$-invariant $2$-partition of $\ZZ^n/L$. By Proposition \ref{prop:matroidCorrespondence}, $\cP$ is the collection of hyperplanes of a $\ZZ^n$-invariant rank-$3$ paving matroid $M$ on $\ZZ^n/L$.
This matroid is the simplification of a $\ZZ^n$-invariant rank-$3$ matroid on $\ZZ^n$, and thus the pair $(L,\cP)$ corresponds to a zero-dimensional degree-$3$ tropical ideal in $\BB[x_1^{\pm 1},\dots,x_n^{\pm 1}]$. 
\end{proof}

\begin{remark}
The proof of Theorem \ref{thm:simplification} can be directly generalized to show that there is a one-to-one correspondence between zero-dimensional degree-$(d+1)$ tropical ideals in $\BB[x_1^{\pm 1},\dots,x_n^{\pm 1}]$ whose polynomials of minimal support have supports of size $2$, $d+1$, or $d+2$, and pairs $(L, \cP)$, where $L$ is a sublattice of $\ZZ^n$ and $\cP$ is a $d$-partition of the quotient group $\ZZ^n/L$ which is invariant under the action of $\ZZ^n$. As in the $d=2$ case, this correspondence sends such a zero-dimensional tropical ideal $I$ of degree $d+1$ to the pair $(L_I, \cH(\si(\uMat(I))))$.
\end{remark}

We conclude this section with a remark about the structure of blocks in a $\ZZ^n$-invariant $d$-partition of a quotient group $\ZZ^n/L$.

\begin{remark}
When $d \leq 2$, the proof of Proposition \ref{prop:latticeBlocks} directly generalizes to invariant $d$-partitions of a quotient group $\Z^n/L$: 

$\bullet$ If $d \leq 2$, $L \subset \ZZ^n$ is a sublattice and $\cP$ is a $\ZZ^n$-invariant $d$-partition of $\ZZ^n/L$, then any block $S \in \cP$ either has the form $S = [\bfv] + K$ with $[\bfv] \in \ZZ^n/L$ and $K$ a nontrivial proper subgroup of $\ZZ^n/L$, or satisfies $|S\cap ([\bfu] + S)| < d$ for all $[\bfu] \in \ZZ^n/L$.

When $d \geq 3$, this dichotomy no longer holds. As an example, suppose $L \subset \ZZ^2$ is the $1$-dimensional sublattice generated by the vector $(2d-2,0)$. 
Consider the subset 
$$S := \{ [(x,y)] \in \ZZ^2/L : x \in \{0,2,\dots,2d-4\}, y \in \{0, 1\}\}.$$
The collection consisting of all subsets of the form $S + [(i,j)]$ with $i \in \{0,1\}$ and $j\in \ZZ$, together with all $d$-subsets not contained in any such  $S + [(i,j)]$, is a $\ZZ^2$-invariant $d$-partition of $\ZZ^2/L$. However, the block $S$ of this $d$-partition is not of the form $S = [\bfv] + K$ with $[\bfv] \in \ZZ^2/L$ and $K$ a nontrivial proper subgroup of $\ZZ^2/L$, but satisfies $[(2,0)] + S = S$.
\end{remark}

\section{Applications}\label{sec:applications}
As discussed in Section \ref{sec:paving}, paving tropical ideals have a relatively 
simple structure encoded by $\ZZ^n$-invariant $d$-partitions of $\ZZ^n$.
The fact that every ``partial'' $\ZZ^n$-invariant $d$-partition can be canonically
completed to a $\ZZ^n$-invariant $d$-partition allows us to develop a notion of 
``generation'' for paving tropical ideals.

\begin{definition}\label{def:generatedPartition}
Suppose $\cA$ is a collection of subsets of $\ZZ^n$ satisfying the following properties:
\begin{enumerate}
    \item[(A1)] $\ZZ^n \notin \cA$.
    \item[(A2)] $|A| \geq d$ for all $A \in \cA$.
    \item[(A3)] If $A_1, A_2\in\cA$ and $\bfu \in \ZZ^n$ satisfy $|A_1 \cap (\bfu+A_2)|\geq d$ then $A_1 = \bfu + A_2$.
\end{enumerate}
In this case, it is not hard to check that $\cA$ can be extended to a $d$-partition of $\ZZ^n$
that is invariant under the action of $\ZZ^n$, defined as
$$\cP_d(\cA):=(\Z^n+\cA) \cup \cD,$$
where
$$\ZZ^n + \cA := \set{\bfu + A : \bfu\in \Z^n \text{ and } A\in \cA}$$
and
$$\cD := \set{S\subset \Z^n : |S|=d \text{ and } S \not\subset X \text{ for all } X \in \Z^n+\cA}.$$
We call $\cP_d(\cA)$ the \textbf{$\ZZ^n$-invariant $d$-partition of $\ZZ^n$ generated by $\cA$}.
Note that Proposition \ref{prop:latticeBlocks} implies that every subset in $\cA$ is either
$d$-sparse or a non-trivial affine sublattice.
\end{definition}

In this section we use this construction to provide several examples and applications
of paving tropical ideals.

\subsection{Uncountably many tropical ideals with Boolean coefficients}
Our first result shows that there are uncountably many paving tropical ideals of degree $3$ with coefficients in $\BB$. 
By contrast, there are only countably many zero-dimensional tropical ideals of degree $2$ with Boolean coefficients, 
as they are in bijection with sublattices of $\Z^n$; see Proposition \ref{prop:degree2}.

\begin{proposition}\label{prop:manyPavingIdeals}
There are uncountably many degree-$3$ paving tropical ideals in $\B[x^{\pm1}]$.
\end{proposition}
\begin{proof}
Fix an integer $m\geq 2$. For $S\subset \N$ with $|S|\geq 2$, consider the set
$$m^S:=\set{m^s : s\in S} \subset \ZZ.$$
It is easy to check that $\cA_S := \{m^S\}$ satisfies axioms (A1), (A2), and (A3) from Definition \ref{def:generatedPartition} with $d = 2$;
indeed, axiom (A3) follows from the fact that $m^S$ is a $2$-sparse subset: 
the value $m^i-m^j$ is uniquely determined by the pair $(i,j)$ if $i \neq j$.
The collection $\cA_S$ thus generates a $\ZZ$-invariant $2$-partition $\cP_2(\cA_S)$ of $\ZZ$,
which by Theorem \ref{thm:correspondenceThm}, is the set of hyperplanes of a degree-$3$
paving tropical ideal in $\BB[x^{\pm 1}]$.
Since the map $S \mapsto \cP_2(\cA_S)$ is injective, there are at least $2^{|\NN|}$
distinct such tropical ideals. 
\end{proof}

\subsection{Extending matroids to tropical ideals}
We now show that it is often possible to extend a paving matroid on a subset of $\ZZ^n$ to a paving tropical ideal.

\begin{theorem}\label{thm:matroidIdeals}
Suppose $E \subset \ZZ^n$ is a (possibly infinite) $d$-sparse subset, and $M$ is a paving matroid of 
degree $d+1$ on the ground set $E$. Then there exists a degree-$(d+1)$ paving tropical ideal 
$I \subset \BB[x_1^{\pm 1},\dots,x_n^{\pm 1}]$ such that $\uMat(I|_E) = M$.
\end{theorem}
\begin{proof}
By Proposition \ref{prop:matroidCorrespondence}, the collection $\cH(M)$ of hyperplanes
of $M$ is a $d$-partition of $E$. In particular, $\cH(M)$ satisfies axioms (A1) and (A2)
of Definition \ref{def:generatedPartition}. To show that $\cH(M)$ satisfies (A3),
suppose $|H_1 \cap (\bfu + H_2)| \geq d$ for $H_1, H_2 \in \cH(M)$ and $\bfu \in \ZZ^n$.
We then have $|E \cap (\bfu + E)| \geq d$, and thus, since $E$ is a $d$-sparse subset, 
$\bfu = \bfzero$. We get $|H_1 \cap H_2| \geq d$, which implies that
$H_1 = H_2$ since $\cH(M)$ is a $d$-partition. This shows that $\cH(M)$ also satisfies
axiom (A3).

We can thus extend $\cH(M)$ to a $\ZZ^n$-invariant $d$-partition $\cP_d(\cH(M))$ 
of $\ZZ^n$, as in Definition \ref{def:generatedPartition}.
By Theorem \ref{thm:correspondenceThm}, there is a degree-$(d+1)$ paving tropical ideal 
$I \subset \BB[x_1^{\pm 1},\dots,x_n^{\pm 1}]$ whose underlying matroid
$\uMat(I)$ has hyperplanes $\cP_d(\cH(M))$.
The matroid $\uMat(I|_E)$ is the restriction of $\uMat(I)$ to the subset $E$, and
thus its hyperplanes are the maximal sets of the form $H \cap E$ with $H \in \cP_d(\cH(M))$.
These are exactly the subsets in $\cH(M)$, as any subset in 
$\cP_d(\cH(M))$ but not in $\cH(M)$ intersects $E$ in less than $d$ elements.
It follows that $\uMat(I|_E) = M$, as desired.
\end{proof}

We note that the previous theorem can be used to easily construct examples of non-realizable tropical ideals.

\begin{example}
The non-Pappus matroid is a non-representable paving matroid of rank $3$ on $9$ elements.
Let $M$ be the non-Pappus matroid on the set $E:=\set{1,2,4,\dots,2^8}\subset \Z$.
Note that $E$ is a $2$-sparse subset of $\Z$.
By Theorem \ref{thm:matroidIdeals}, there is a degree-$3$ paving tropical ideal 
$I \subset \B[x_1^{\pm 1}]$ such that $\uMat(I|_E) = M$. 
In particular, this implies that $I$ is a non-realizable tropical ideal.
\end{example}

\subsection{A non-realizable zero-dimensional tropical ideal of degree 2 with Boolean coefficients}

It was shown in \cite{Zaj18} that all zero-dimensional degree-$2$ tropical ideals in $\BB[x^{\pm 1}]$ are 
realizable. Below we show that this is not the case for zero-dimensional degree-$2$ tropical ideals
in more variables.

We first state the following result about the restriction of a paving tropical ideal
to a subsemiring of tropical polynomials in fewer variables.

\begin{proposition}\label{prop:restriction}
Suppose $I$ is a degree-$(d+1)$ paving tropical ideal in the semiring $\B[x_1^{\pm1},\dots,x_n^{\pm1},y_1^{\pm 1},\dots, y_m^{\pm 1}]$, corresponding to the $\ZZ^{n+m}$-invariant $d$-partition $\cP$ of $\ZZ^{n+m}$.
Identify $\ZZ^n$ with the sublattice $\ZZ^n \times \{\bfzero\} \subset \ZZ^{n+m}$ in the natural way. 
If there is no $S\in \cP$ such that $\ZZ^n \subset S$ then $I \cap \B[x_1^{\pm1},\dots,x_n^{\pm1}]$ is a degree-$(d+1)$ paving tropical ideal,
corresponding to the $\ZZ^{n}$-invariant $d$-partition of $\ZZ^n$
$$\cP|_{\ZZ^n} := \{ S \cap \ZZ^n : S \in \cP \text{ and } |S \cap \ZZ^n| \geq d \}.$$
If there is $S\in \cP$ such that $\ZZ^n \subset S$ then $I \cap \B[x_1^{\pm1},\dots,x_n^{\pm1}]$ is the degree-$d$ paving tropical ideal corresponding to the uniform $(d-1)$-partition of $\ZZ^n$ consisting
of all subsets of size $d-1$.
\end{proposition}
\begin{proof}
Denote $I' = I \cap \B[x_1^{\pm1},\dots,x_n^{\pm1}]$. The underlying matroid $\uMat(I')$ is the
restriction of the paving matroid $\uMat(I)$ to $\ZZ^n$, and so it is a paving matroid.
The collection $\cP$ is the set of hyperplanes of $\uMat(I)$.
If there is a hyperplane $S\in \cP$ such that $\ZZ^n \subset S$ then $\uMat(I')$ is the uniform matroid 
of rank $d$, which corresponds to the uniform $(d-1)$-partition of $\ZZ^n$.
If there is no hyperplane $S\in \cP$ such that $\ZZ^n \subset S$ then $\uMat(I')$ has rank $d+1$, 
and its hyperplanes are the maximal subsets of the form $S \cap \ZZ^n$ with $S \in \cP$.
As $\cP$ is a $d$-partition, these are exactly the subsets of the form $S \cap \ZZ^n$ with $S \in \cP$
and $|S \cap \ZZ^n| \geq d$, as claimed. 
\end{proof}

Recall that zero-dimensional degree-$2$ tropical ideals with Boolean coefficients are in bijection with proper sublattices
of $\ZZ^n$, as in Proposition \ref{prop:degree2}.

\begin{lemma}\label{lem:char2}
The zero-dimensional degree-$2$ tropical ideal $I \subset \BB[x^{\pm 1}]$ corresponding to the sublattice $4\Z \subset \ZZ$ is not realizable over any field of characteristic $2$.
\end{lemma}
\begin{proof}
Suppose for a contradiction that $I = \trop(J)$ for an ideal $J \subset K[x^{\pm 1}]$ with $K$ a field of characteristic $2$. Let $J' = J \cap K[x]$. Note that $I' := I \cap \BB[x] = \trop(J')$.
Since $K[x]$ is a principal ideal domain and $x^2 \tplus x \tplus 0 \in I'$ is the unique polynomial of minimal degree, 
we must have $J' =\ideal{x^2+a x+b}$ for some nonzero $a,b \in K$. 
The tropical ideal $I'$ also contains the polynomial $x^4\tplus 0$,
so $x^4 + e \in J'$ for some nonzero $e\in K$. This implies that there exist $c,d \in K$ such that
\begin{align*}
    x^4+e &=(x^2+ax+b)(x^2+cx+d)\\
            &=x^4+(a+c) x^3+ (ac+b+d)x^2+(ad+bc)x+bd.
\end{align*}
It follows that $a=c$, $a^2+b+d = 0$, and $ad+ba=0$. The last equality implies $d+b=0$,
and so we get $a^2=0$ from the second equality, a contradiction. 
\end{proof}

\begin{proposition}\label{prop:nonrealizable}
The zero-dimensional degree-$2$ tropical ideal $I \subset \BB[x_1^{\pm 1}, x_2^{\pm 1}, x_3^{\pm 1}]$ 
corresponding to the sublattice $L \subset \ZZ^3$ generated by $(4,0,0), (0,2,0), (0,0,2)$ is not realizable.
\end{proposition}

\begin{proof}
Suppose $I=\trop(J)$ for some ideal $J \subset  K[x_1^{\pm 1}, x_2^{\pm 1}, x_3^{\pm 1}]$ with coefficients
in a field $K$. 
We then have $$I^{\{1\}}:=I\cap \B[x_1^{\pm 1}]= \trop(J\cap K[x_1^{\pm 1}])$$
and  
$$I^{\set{2,3}}:=I\cap \BB[x_2^{\pm 1},x_3^{\pm 1}]=\trop(J\cap K[x_2^{\pm 1},x_3^{\pm 1}]).$$
The partition of $\ZZ^3$ corresponding to $I$ is the collection of all translates of $L$.
By Proposition \ref{prop:restriction}, the ideal $I^{\{1\}} \subset \BB[x_1^{\pm 1}]$ is the degree-$2$ paving tropical ideal
corresponding to the partition of $\ZZ$ consisting of all the translates of $L \cap \ZZ = 4\ZZ$, 
so it is the zero-dimensional degree-$2$ tropical ideal corresponding to the sublattice 
$4\ZZ \subset \ZZ$. It follows by Lemma \ref{lem:char2} that $K$ is a field of characteristic not equal to $2$.
In a similar way, $I^{\set{2,3}}$ is the zero-dimensional degree-$2$ tropical ideal corresponding to the
sublattice of $\ZZ^2$ generated by $(2,0)$ and $(0,2)$. By \cite{Zaj18}*{Proposition 5.2.8}, this
tropical ideal is not realizable over a field of characteristic other than $2$, a contradiction.
\end{proof}

It would be interesting to know whether all zero-dimensional degree-$2$ tropical ideals in $\BB[x_1^{\pm 1}, x_2^{\pm 1}]$ are realizable over some field.

\bibliographystyle {alpha}
\bibliography {main}

\end{document}